\documentclass[12pt]{amsart}

\usepackage{amsmath,amsthm,amssymb}
\usepackage{geometry}
\usepackage{color}
\usepackage{url}

\newtheorem{lm}{Lemma}

\newtheorem{prop}{Proposition}

\newtheorem{thm}{Theorem}

\newtheorem{cor}{Corollary}

\theoremstyle{definition}
\newtheorem{dfn}{Definition}

\newtheorem{rmk}{Remark}

 \author[Franzina]{Giovanni Franzina}
 \address{Dipartimento di Matematica, Universit\`a degli Stud\^i di Trento, Via Sommarive 14, Povo (TN), Italy}
 \email{franzina@science.unitn.it}
 
 \author[Lindqvist]{Peter Lindqvist}
 \address{Department of Matematics, Norwegian University of Science and 
Technology, 7491 Trondheim, Norway}
 \email{lqvist@math.ntnu.no}
\title[An eigenvalue problem with variable exponents]{An eigenvalue problem with variable exponents}

\begin{document}
\maketitle

\begin{abstract}
 A highly nonlinear eigenvalue problem is studied in a Sobolev space with 
variable exponent. The Euler-Lagrange equation for the minimization of a 
Rayleigh quotient of two Luxemburg norms is derived. The asymptotic case 
with a ``variable infinity'' is treated. Local uniqueness is proved for 
the viscosity solutions.
\end{abstract}

\section{Introduction}
An expedient feature of many eigenvalue problems is that the eigenfunctions may be multiplied by constants.
That is the case for our non-linear problem in this note. We will study the eigenvalue problem coming from 
the minimization of the Rayleigh quotient
\begin{equation}\label{eq:luxray}
	\frac{\|\nabla u \|_{p(x),\Omega}}{\| u \|_{p(x),\Omega}}
\end{equation}
among all functions belonging to the Sobolev space $W^{1,p(x)}_{0}(\Omega)$ with variable exponent $p(x)$.
Here $\Omega$ is a bounded domain in $\mathbb{R}^{n} $ and the variable exponent $p(x)$ is a smooth function,
$1<p^{-}\le p(x)\le p^{+}<\infty$. The norm is the so-called {\em Luxemburg norm}.
\vskip.2cm

If $p(x)=p$, a constant in the range $1<p<\infty$, the problem reduces to the minimization of the Rayleigh quotient
\begin{equation}\label{eq:rayp}
	\frac{\displaystyle \int_{\Omega} |\nabla u|^{p}\,dx }{\displaystyle\int_{\Omega}|u|^{p}\,dx}
\end{equation}
among all $u\in W^{1,p}_{0}(\Omega)$, $u\not\equiv0$. Formally, the Euler-Lagrange equation is 
\begin{equation}\label{eq:typicaleig}
	\mathop{\rm div} \big( |\nabla u|^{p-2}\nabla u \big) + \lambda |u|^{p-2}u=0.
\end{equation}
The special case $p=2$ of this much studied problem yields the celebrated Helmholtz equation
\[
	\Delta u + \lambda u = 0.
\]

It is decisive that homogeneity holds: if $u$ is a minimizer, so is $cu$ for any non-zero constant $c$. On the contrary,
the quotient
\begin{equation}\label{eq:raypx}
\frac{\displaystyle \int_{\Omega}|\nabla u|^{p(x)}\,dx}{\displaystyle\int_{\Omega}|u|^{p(x)}\,dx}
\end{equation}
with variable exponent does not possess this expedient property, in general. Therefore its infimum
over all $\varphi\in C^{\infty}_{0}(\Omega)$, $\varphi\not\equiv0$, is often zero and no mimizer appears in the space
$W^{1,p(x)}_{0}(\Omega)$, except the trivial $\varphi\equiv0$, which is forbidden. For an example, we refer to \cite[pp. 444--445]{FZ2}. A way to avoid this collapse is to impose the constraint
\[
	\int_{\Omega}|u|^{p(x)}\,dx= \, \text{constant}. 
\]

Unfortunately, in this setting the minimizers obtained for different normalization constants are difficult to compare
in any reasonable way, except, of course, when $p(x)$ is constant.
For a suitable $p(x)$, it can even happen that any positive $\lambda$ is an eigenvalue for some choice of the
normalizing constant. Thus \eqref{eq:raypx} is not a proper generalization of \eqref{eq:rayp}, which
has a well defined spectrum.

A way to avoid this situation is to use the Rayleigh quotient \eqref{eq:luxray}, where we have used the notation
\begin{equation}\label{eq:luxnorm}
\| f\|_{p(x),\Omega} = \inf \left\{ \gamma>0\,\colon\, \int_{\Omega}\left\vert \frac{f(x)}{\gamma}\right\vert^{p(x)}
					\frac{dx}{p(x)}\le 1	\right\}
\end{equation}
for the Luxemburg norm.
This restores the homogeneity. In the integrand, the use of $p(x)^{-1}\,dx$ (rather than $dx$) has no
bearing, but it simplifies the equations a little.

The existence of a minimizer follows easily by the direct method in the Calculus of Variations, cf.~\cite{G}.
We will derive the Euler-Lagrange equation
\begin{equation}
\label{eq:EL}
\mathop{\rm div} \left( \left\vert \frac{\nabla u}{K}\right\vert^{p(x)-2} \frac{\nabla u}{K}\right)
	+ \frac{K}{k}S \left\vert \frac{u}{k}\right\vert^{p-2}\frac{u}{k}=0,
\end{equation}
where the constants are
\begin{equation}\label{eq:const}
K = \| \nabla u\|_{p(x)}, \quad k = \| u \|_{p(x)}, \quad
	S = \frac{\displaystyle \int_{\Omega}\left\vert \frac{\nabla u}{K}\right\vert^{p(x)}\,dx}{
		\displaystyle \int_{\Omega}\left\vert\frac{u}{k}\right\vert^{p(x)}\,dx}.
\end{equation}
They depend on $u$. Notice that we are free to fix only one of the norms $K$ and $k$. The minimum of the
Rayleigh quotient \eqref{eq:luxray} is $\frac{K}{k}$. Inside the integrals defining the constant $S$,
we now have $dx$ (and not $p(x)^{-1}\,dx$), indeed.
Therefore it is possible that $S\neq1$. For a constant exponent, $S=1 $ and the Euler-Lagrange equation
reduces to \eqref{eq:typicaleig}. Sometimes we write $K_{u},k_{u}, S_{u}$.
\vskip.2cm

We are interested in replacing $p(x)$ by a {\em ``variable infinity''} $\infty(x)$. The passage to infinity is
accomplished so that $p(x)$ is successively replaced by $jp(x)$, $j=1,2,3\ldots$ In order to identify the limit equation,
as $jp(x)\to\infty$, we use the theory of viscosity solutions. In the case of a constant $p(x)$, the limit
equation is
\begin{equation}
\label{eq:lqvinfty}
\max\left\{ \Lambda_{\infty}- \frac{|\nabla u|}{u},\sum_{i,j=1}^{n} \frac{\partial u}{\partial x_{i}}\frac{\partial u}{\partial x_{j}} \frac{\partial^{2} u}{\partial x_{i}\partial x_{j}}  \right\}=0,
\end{equation}
where $u\in W^{1,\infty}_{0}(\Omega)$, $u>0$ and
\begin{equation}
\label{eq:inradius}
\Lambda_{\infty} = \frac{1}{\displaystyle \max_{x\in \Omega} \mathop{\rm dist}(x,\partial\Omega)}.
\end{equation}
This has been treated in~\cite{JLM} (see also~\cite{CD,JL,J}). An interesting interpretation in terms of optimal mass transportation is given
in~\cite{CDJ}. According to a recent manuscript by Hynd, Smart and Yu, there are domains in which
there can exist several linearly independent positive eigenfunctions, see~\cite{HSY}. Thus the eigenvalue
$\Lambda_{\infty}$ is not always simple.

In our case the limit equation reads
\begin{equation}
\label{eq:curiouseig}
\max\left\{ \Lambda_{\infty} - \frac{|\nabla u|}{u}, \:\Delta_{\infty(x)}\! \left(\,\frac{u}{K}\,\right) \right\} = 0,
\end{equation}
where $K=\|\nabla u\|_{\infty,\Omega}$ and
\begin{equation}
\label{eq:inftyx}
\Delta_{\infty(x)} v = \sum_{i,j=1}^{n} \frac{\partial v}{\partial x_{i}}\frac{\partial v}{\partial x_{j}} \frac{\partial^{2} v}{
	\partial x_{i}\partial x_{j}} 
	+ |\nabla v|^{2}\,\ln\big(|\nabla v|\big) \big\langle  \nabla v,\nabla \ln p \big\rangle.
\end{equation}
We are able to establish that if $\Lambda_{\infty}$ is given the value \eqref{eq:inradius}, the same as for a
constant exponent, then the existence of a non-trivial solution is guaranteed. We also prove
a {\em local uniqueness} result: in very small interior subdomains we cannot ``improve'' the solution. The technically
rather demanding proof is based on the modern theory of viscosity solutions, cf.~\cite{CIL,K}, and
we assume that the reader is familiar with this topic.
\vskip.2cm

Needless to say, many open problems remain. To mention one, for a finite variable exponent $p(x)$ we
do not know whether
the first eigenvalue (the minimum of the Rayleigh quotient) is simple. The methods in~\cite{BK,L} do not work well
now. There are also many gaps in the theory available at present: due to the lack of a proper Harnack inequality, we cannot 
assure that the limit of the $jp(x)$-eigenfunctions is strictly positive. A discussion about analogous
difficulties can be found in~\cite{AH}.
In the present work we restrict ourselves to positive eigenfunctions.
We hope to return to this fascinating topic in
the future.

\section{Preliminaries}
We will always assume that $\Omega$ is a bounded
domain in the $n$-dimensional space $\mathbb{R}^{n}$ and that the variable
exponent $p(x)$ is in the range
\begin{equation}
\label{eq:range}
1<p^{-}\le p(x)\le p^{+}<\infty,
\end{equation}
when $x\in\Omega$, and belongs to $C^{1}(\Omega)\cap W^{1,\infty}(\Omega)$. Thus
$\| \nabla p \|_{\infty,\Omega}<\infty$.

Next we define $L^{p(x)}(\Omega)$ and the Sobolev space $W^{1,p(x)}(\Omega)$ with variable exponent $p(x)$.
We say that $u\in W^{1,p(x)}(\Omega)$ if $u$ and its distributional gradient $\nabla u$ are
measurable functions satisfying
\[
	\int_{\Omega}|u|^{p(x)}\,dx<\infty,\quad \int_{\Omega}|\nabla u|^{p(x)}\,dx<\infty.
\]
The norm of the space $L^{p(x)}(\Omega)$ is defined by \eqref{eq:luxnorm}. This is a Banach space. So is
$W^{1,p(x)}(\Omega)$ equipped with the norm
\[
	\|u \|_{p(x),\Omega} + \| \nabla u \|_{p(x),\Omega}.
\]
Smooth functions are dense in this space, and so we can define the space $W^{1,p(x)}_{0}(\Omega)$
as the completion of $C^{\infty}_{0}(\Omega)$ in the above norm. We refer to~\cite{FZ2} and the monograph~\cite{DHHR}
about these spaces.
\vskip.2cm
The following properties are used later.

\begin{lm}[Sobolev]\label{lm:sob}
The inequality
\[
 \| u\|_{p(x),\Omega} \le C \| \nabla u\|_{p(x),\Omega}
\]
holds for all $u\in W^{1,p(x)}_0(\Omega)$; the constant is independent of $u$.
\end{lm}

In fact, even a stronger inequality is valid.

\begin{lm}[Rellich-Kondrachev]
 \label{lm:RellKondr}
\sloppy Given a sequence $u_j\in  W^{1,p(x)}_0(\Omega) $ such that $\|\nabla u_j \|_{p(x),\Omega}\le M$,
$j=1,2,3,\ldots,$ there exists a $u\in W^{1,p(x)}_0(\Omega)$ such that
$u_{j_\nu} \to u$
strongly in $L^{p(x)}(\Omega)$ and
$\nabla u_{j_\nu}\rightharpoonup \nabla u$ weakly in $L^{p(x)}(\Omega)$ for some subsequence.
\end{lm}

\vskip.2cm
Eventually, from now on we shall write $\|\cdot\|_{p(x)}$ rather than $\|\cdot\|_{p(x),\Omega}$, provided
this causes no confusion.

\vskip.2cm
We need to identify the space
\(
	\bigcap_{j=1}^{\infty} W^{1,jp(x)}(\Omega).
\)
This limit space is nothing else than the familiar $W^{1,\infty}(\Omega)$. According to the next lemma,
it is independent of $p(x)$.

\begin{lm}
\label{lm:elem}
If $u$ is a measurable function in $\Omega$, then
\[
	\lim_{j\to\infty} \|u\|_{jp(x)} = \| u \|_{\infty}.
\]
\end{lm}

\begin{proof}
The proof is elementary. We use the notation
\begin{align*}
	M &  = \| u\|_{\infty} = \mathop{\rm ess\,sup}_{x\in\Omega} |u(x)|,  & \\
	 M_{j} & = \| u \|_{jp(x)},
\end{align*}
and we claim that
\[
	\lim_{j\to\infty} M_{j} = M.
\]

To show that $\limsup_{j\to\infty} M_{j}\le M$, we only have to consider those indices $j$ for which $M_{j}>M$.
Then, since $p(x)>1$,
\[
1 = \left(\int_{\Omega}\left\vert \frac{u(x)}{M_{j}}\right\vert^{jp(x)}\,\frac{dx}{jp(x)}\right)^\frac{1}{j}
\le \frac{M}{M_{j}} \left(\int_{\Omega} \frac{dx}{jp(x)}\right)^\frac{1}{j},
\]
and the inequality follows.

To show that $\liminf_{j\to\infty} M_{j}\ge M$, we may assume that $0<M<\infty$.
Given $\varepsilon>0$, there is a set
$A_{\varepsilon}\subset\Omega$ such that $\text{\rm meas} (A_{\varepsilon})>0$ and $|u(x)|>M-\varepsilon$
in $A_{\varepsilon}$. We claim that $\liminf_{j\to\infty}M_{j}\ge M-\varepsilon$. Ignoring those indices for which
$M_{j}\ge M-\varepsilon$, we have
\[
1=\left(\int_{\Omega} \left\vert \frac{u(x)}{M_{j}}\right\vert^{jp(x)}  \frac{dx}{jp(x)}\right)^\frac{1}{j}
\ge \left(\int_{A_{\varepsilon}} \left\vert \frac{u(x)}{M_{j}}\right\vert^{jp(x)} \, \frac{dx}{jp(x)}\right)^\frac{1}{j}
\ge \frac{M-\varepsilon}{M_{j}} \left(\int_{A_{\varepsilon}} \frac{dx}{jp^{+}}\right)^\frac{1}{j},
\]
and the claim follows. Since $\varepsilon$ was arbitrary the Lemma follows. The case $M=\infty$
requires a minor modification in the proof.
\end{proof}

\section{The Euler Lagrange equation}
We define
\begin{equation}
\label{eq:princfreq}
\Lambda_{1} = \inf_{v} \frac{\|\nabla v\|_{p(x)}}{\|v\|_{p(x)}},
\end{equation}
where the infimum is taken over all $v\in W^{1,p(x)}_{0}(\Omega)$, $v\not\equiv0$. One gets the same
infimum by requiring that $v\in C^{\infty}_{0}(\Omega)$.
The Sobolev inequality (Lemma~\ref{lm:sob})
\[
	\| v\|_{p(x)}\le C \|\nabla v\|_{p(x)},
\]
where $C$ is independent of $v$, shows that $\Lambda_{1}>0$. 

To establish the
existence of a non-trivial minimizer, we select a minimizing sequence of admissible functions $v_{j}$,
normalized so that $\|v_{j}\|_{p(x)}=1$. Then $$\Lambda_{1}=\lim_{j\to\infty}\|\nabla v_{j}\|_{p(x)}.$$
Recall the Rellich-Kondrachev Theorem for Sobolev spaces with variable exponents (Lemma~\ref{lm:RellKondr}).
Hence, we can extract a subsequence
$v_{j_{\nu}}$ and find a function $u\in W^{1,p(x)}_{0}(\Omega)$ such that $v_{j_{\nu}}\to u$ strongly in
$L^{p(x)}(\Omega)$ and $\nabla v_{j_{\nu}}\rightharpoonup \nabla u$ weakly in $L^{p(x)}(\Omega)$.
The norm is weakly sequentially lower semicontinuous. Thus,
\[
	\frac{\|\nabla u\|_{p(x)}}{\| u\|_{p(x)}}
	\le \lim_{\nu\to\infty}
	\frac{\|\nabla v_{j_{\nu}}\|_{p(x)} }{\|v_{j_{\nu}}\|_{p(x)}} = \Lambda_{1}.
\]
This shows that $u$ is a minimizer. Notice that if $u$ is a minimizer, so is $|u|$. We have proved
the following proposition.

\begin{prop}
There exists a non-negative minimizer $u\in W^{1,p(x)}_{0}(\Omega)$, $u\not\equiv0$, of the Rayleigh
quotient \eqref{eq:luxray}.
\end{prop}

In order to derive the Euler-Lagrange equation for the minimizer(s), we fix an arbitrary test function
$\eta\in C^{\infty}_{0}(\Omega)$ and consider the competing function
\[
	v(x)=u(x)+\varepsilon \eta(x),
\]
and write
\[
	k=k({\varepsilon})=\|v\|_{p(x)}, \quad K = K(\varepsilon) = \|\nabla v\|_{p(x)}.
\]
A necessary condition for the inequality
\[
	\Lambda_{1} = \frac{K(0)}{k(0)} \le \frac{K(\varepsilon)}{k(\varepsilon)}
\]
is that
\[
\frac{d}{d\varepsilon} \left( \frac{K(\varepsilon)}{k(\varepsilon)}\right) 
	= \frac{K'(\varepsilon)k(\varepsilon)-K(\varepsilon)k'(\varepsilon)}{k(\varepsilon)^{2}} = 0,
		\quad \text{for   } \varepsilon=0.
\]
Thus the necessary condition of minimality reads
\begin{equation}\label{eq:ELkK}
	\frac{K'(0)}{K(0)} = \frac{k'(0)}{k(0)}.
\end{equation}
To find $K'(0)$, differentiate
\[
\int_{\Omega} \left\vert \frac{ \nabla u(x) + \varepsilon \nabla \eta (x) }{K(\varepsilon)}\right\vert^{p(x)} \,
	\frac{dx}{p(x)} = 1,
\]
with respect to $\varepsilon$.
Differentiation under the integral sign is justifiable.  We obtain
\[
	\int_{\Omega} \frac{|\nabla u+\varepsilon \nabla \eta |^{p(x)-2}
	\big\langle \nabla u + \varepsilon\nabla \eta,\nabla\eta\big\rangle}{
		K(\varepsilon)^{p(x)}}\,dx=\int_{\Omega} \frac{|\nabla u+\varepsilon\nabla\eta|^{p(x)}}{
		K(\varepsilon)^{p(x)+1}}K'(\varepsilon)\,dx.
\]
For $\varepsilon=0$, we conclude
\[
\frac{K'(0)}{K(0)} = \frac{\displaystyle\int_{\Omega} K^{-p(x)}
	|\nabla u|^{p(x)-2}\big\langle \nabla u,\nabla\eta\big\rangle\,dx}{
	\int_{\Omega}\left\vert \frac{\nabla u}{K} \right\vert^{p(x)}\,dx}.
\]
A similar calculation yields
\[
\frac{k'(0)}{k(0)} = \frac{\displaystyle \int_{\Omega} k^{-p(x)} |u|^{p(x)-2}u\,\eta\,dx}{
	\int_{\Omega} \left\vert \frac{u}{k}\right\vert^{p(x)}\,dx}\,.
\]
Inserting the results into \eqref{eq:ELkK}, we arrive at equation \eqref{eq:EL} in weak form: for all
test functions $\eta \in C^{\infty}_{0}(\Omega)$, we have
\[
\int_{\Omega}\left\vert\frac{\nabla u}{K}\right\vert^{p(x)-2}\left\langle\frac{\nabla u}{K},\nabla\eta\right\rangle\,dx
=\Lambda_{1}S\int_{\Omega}\left\vert\frac{u}{k}\right\vert^{p(x)-2}\frac{u}{k}\,\eta\,dx,
\]
where $K=\|\nabla u\|_{p(x)}$, $k=\|u\|_{p(x)}$ and $S$ is as in \eqref{eq:const}. Here $\Lambda_{1}=K/k$.
\vskip.2cm

The weak solutions with zero boundary values are called {\em eigenfunctions}, except $u\equiv0$.
We refer to~\cite{AM,AM2,F,FZ,HHLT} for regularity theory.

\begin{dfn}
A function $u\in W^{1,p(x)}_{0}(\Omega)$, $u\not\equiv0$, is an {\em eigenfunction} if the equation
\begin{equation}
	\label{eq:eigeq}
	\int_{\Omega}\left\vert\frac{\nabla u}{K}\right\vert^{p(x)-2}\left\langle\frac{\nabla u}{K},\nabla\eta\right\rangle\,dx
=\Lambda S\int_{\Omega}\left\vert\frac{u}{k}\right\vert^{p(x)-2}\frac{u}{k}\,\eta\,dx
\end{equation}
holds whenever $\eta\in C^{\infty}_{0}(\Omega)$. Here $K=K_{u}$, $k=k_{u}$ and $S=S_{u}$. The corresponding
$\Lambda$ is the {\em eigenvalue}.
\end{dfn}

\begin{rmk}
 \label{rmk:continuity}
According to \cite{AM2,FZ,F}, the weak solutions of equations like \eqref{eq:eigeq} are continuous if the variable
exponent $p(x)$ is H\"older continuous. Thus the eigenfunctions are continuous.
\end{rmk}

If $\Lambda_{1}$ is the minimum of the Rayleigh quotient in \eqref{eq:princfreq}, we must have
\[
	\Lambda\ge \Lambda_{1},
\]
in \eqref{eq:eigeq}, thus $\Lambda_{1}$ is called the {\em first eigenvalue} and the corresponding eigenfunctions
are said to be {\em first eigenfunctions}. To see this, take $\eta=u$ in the equation, which is possible by approximation. Then we obtain, upon cancellations, that
\[
\Lambda = \frac{K}{k} = \frac{\|\nabla u\|_{p(x)}}{\|u\|_{p(x)}} \ge \Lambda_{1}.
\]
We shall restrict ourselves to positive eigenfunctions.

\begin{thm}
\label{thm:exfirsteig}
There exists a continuous strictly positive first eigenfunction. Moreover, any non-negative eigenfunction is strictly positive.
\end{thm}

\begin{proof}
The existence of a first eigenfunction was clear, since minimizers of \eqref{eq:princfreq} are solutions
of \eqref{eq:eigeq}. But if $u$ is a minimizer, so is $|u|$, and $|u|\ge0$. Thus we have a non-negative one.
By Remark~\ref{rmk:continuity} the eigenfunctions are continuous. The strict positivity then follows by
the strong minimum principle for weak supersolutions in~\cite{HHLT}.
\end{proof}

\vskip.2cm

We are interested in the asymptotic case when the variable exponent approaches $\infty$ via the sequence
$p(x)$, $2p(x)$, $3p(x)\ldots$ The procedure requires viscosity solutions. Thus we first verify that the weak solutions
of the equation \eqref{eq:eigeq}, formally written as
\begin{equation*}
\mathop{\rm div} \left( \left\vert \frac{\nabla u}{K}\right\vert^{p(x)-2} \frac{\nabla u}{K}\right)
	+\Lambda \,S \left\vert \frac{u}{k}\right\vert^{p-2}\frac{u}{k}=0,
\end{equation*}
are viscosity solutions. Given $u\in C(\Omega)\cap W^{1,p(x)}_{0}(\Omega)$, {\em we fix the parameters}
$k=\|u\|_{p(x)}$, $K=\|\nabla u\|_{p(x)}$
{\em and} $S$. Replacing $u$ by a function $\phi\in C^{2}(\Omega)$, but keeping $k,K,S$ unchanged,
we formally get
\[
	\Delta_{p(x)}\phi - |\nabla\phi|^{2}\log(K)\big\langle \nabla \phi,\nabla p(x)\big\rangle
		+\Lambda^{p(x)}\,S|\phi|^{p(x)-2}\phi=0,
\]
where
\begin{align*}
	\Delta_{p(x)}\phi  = \mathop{\rm div} \big( |\nabla \phi|^{p(x)-2}\nabla\phi\big)
	 = |\nabla\phi|^{p(x)-4}\Big\{ |\nabla\phi|^{2}  \Delta & \phi + \big(p(x)-2\big) \Delta_{\infty}\phi
	& \\ & +|\nabla\phi|^{2}\ln\big(|\nabla\phi|\big)\big\langle\nabla\phi,\nabla p(x)\big\rangle \Big\}, &
\end{align*}
and
\[
	\Delta_{\infty}\phi = \sum_{i,j=1}^{n}\frac{\partial \phi}{\partial x_{i}} \frac{\partial \phi}{\partial x_{j}}
							\frac{\partial^{2}\phi}{\partial x_{i}\partial x_{j}}
\]
is the $\infty$-Laplacian. The relation $\Lambda = K/k$ was used in the simplifications.

Let us abbreviate the expression as
\begin{align}
F(x,\phi,\nabla\phi, & D^{2}\phi) = & \nonumber \\
 & |\nabla\phi|^{p(x)-4} \Big\{ |\nabla\phi|^{2}  \Delta \phi + \big(p(x)-2\big) \Delta_{\infty}\phi
 +|\nabla\phi|^{2}\ln\big(|\nabla\phi|\big)\big\langle\nabla\phi,\nabla p(x)\big\rangle & \nonumber \\
 & \qquad \qquad  \qquad  - |\nabla\phi|^{2}\log(K)\big\langle \nabla \phi,\nabla p(x)\big\rangle \Big\}
  + \Lambda^{p(x)} \, S |\phi|^{p(x)-2}\phi =0.
\end{align}
where we deliberately take $p(x)\ge2$. Notice that
\[
	F(x,\phi,\nabla\phi,D^{2}\phi) <0
\]
exactly when
\[
	\Delta_{p(x)}\! \left( \frac{\phi}{K}\right) + \Lambda\, S\left|\frac{\phi}{k}\right|^{p(x)-2}
  \frac{\phi}{k}<0.
\]
Recall that $k,K,S$ where dictated by $u$.

Let $\phi\in C^{2}(\Omega)$ and $x_{0}\in \Omega$. We say that $\phi\in C^{2}(\Omega)$ touches $u$ from below at the 
point $x_{0}$, if $\phi(x_{0})=u(x_{0})$ and $\phi(x)<u(x)$ when $x\neq x_{0}$.

\begin{dfn}
Suppose that $u\in C(\Omega)$. We say that $u$ is a {\em viscosity supersolution} of the equation
\[
	F(x,u,\nabla u,D^{2}u)=0
\]
if, whenever $\phi$ touches $u$ from below at a point $x_{0}\in \Omega$, we have
\[
F(x_{0},\phi(x_{0}),\nabla\phi(x_{0}),D^{2}\phi(x_{0}))\le 0.
\]
We say that $u$ is a {\em viscosity subsolution} if, whenever $\psi\in C^{2}(\Omega)$ touches $u$ from above
at a point $x_{0}\in\Omega$, we have
\[
	F(x_{0},\psi(x_{0}),\nabla\psi(x_{0}),D^{2}\psi(x_{0}))\ge0.
\] 
Finally, we say that $u$ is a {\em viscosity solution} if it is both a viscosity super- and subsolution.
\end{dfn}

Several remarks are appropriate. Notice that the operator $F$ is evaluated for the test function and only at the 
touching point. If the family of test functions is empty at some point, then there is no requirement on $F$ at that point.
The definition makes sense for a merely continuous function $u$, provided that the parameters $k,K,S,\Lambda$
have been assigned values. We always have $\nabla u$ available for this in our problem.

\begin{thm}
\label{thm:WareV}
The eigenfunctions $u$ are viscosity solutions of the equation
\[
	F(x,u,\nabla u,D^{2}u)=0.
\]
\end{thm}
\begin{proof}
This is a standard proof. The equation
\begin{equation}\label{eq:weak}
\int_{\Omega}\left\vert\frac{\nabla u}{K}\right\vert^{p(x)-2}\left\langle\frac{\nabla u}{K},\nabla\eta\right\rangle\,dx
=\Lambda S\int_{\Omega}\left\vert\frac{u}{k}\right\vert^{p(x)-2}\frac{u}{k}\,\eta\,dx
\end{equation}
holds for all $\eta\in W^{1,p(x)}_{0}(\Omega)$. We first claim that $u$ is a viscosity supersolution. Our proof is
indirect. The {\em antithesis} is that there exist a point $x_{0}\in\Omega$ and a test function 
$\phi\in C^{2}(\Omega)$, touching  $u$ from below at $x_{0}$, such that $F(x_{0},\phi(x_{0}),\nabla\phi(x_{0}),
D^{2}\phi(x_{0}))>0$. By continuity, 
\[
	F(x,\phi(x),\nabla\phi(x),D^{2}\phi(x))>0
\]
holds when $x\in B(x_{0},r)$ for some radius $r$ small enough. Then also
\begin{equation}\label{eq:testeq}
\Delta_{p(x)} \left( \frac{ \phi(x)}{K}\right)
	+\Lambda \,S \left\vert \frac{\phi(x)}{k}\right\vert^{p-2}\frac{\phi(x)}{k}>0,
\end{equation}
in $B(x_{0},r)$. Denote
\[
	\varphi = \phi + \frac{m}{2},\qquad m = \min_{\partial B(x_{0},r)}(u-\phi).
\]
Then $\varphi<u$ on $\partial B(x_{0},r)$ but $\varphi(x_{0}) > u(x_{0})$, since $m>0$. Define
\[
	\eta = \big[ \varphi-u\big]_{+} \chi_{B(x_{0},r)}.
\]
Now $\eta\ge0$. If $\eta\not\equiv 0$, we multiply \eqref{eq:testeq} by $\eta$ and we integrate by parts to obtain
the inequality
\[
\int_{\Omega}\left\vert\frac{\nabla \phi}{K}\right\vert^{p(x)-2}\left\langle\frac{\nabla \phi}{K},\nabla\eta\right\rangle\,dx
<\Lambda S\int_{\Omega}\left\vert\frac{\phi}{k}\right\vert^{p(x)-2}\frac{\phi}{k}\,\eta\,dx
\]
We have $\nabla \eta = \nabla\phi-\nabla u$ in the subset where $\varphi\ge u$. Subtracting equation \eqref{eq:weak} by the 
above inequality, we arrive at
\begin{align*}
\int_{\{\varphi>u\}}& \left\langle  \left\vert\frac{\nabla \phi}{K}\right\vert^{p(x)-2}\frac{\nabla \phi}{K}�
-\left\vert\frac{\nabla u}{K}\right\vert^{p(x)-2}\frac{\nabla u}{K} ,
\frac{\nabla \phi}{K} - \frac{\nabla u}{K}
\right\rangle\! dx & \\
& < S \int_{\{\varphi>u\}} \left( \left\vert \frac{\phi}{k} \right\vert^{p(x)-2}\frac{\phi}{k} - \left\vert \frac{u}{k} \right\vert^{p(x)-2}
	\frac{u}{k}\right)\!\! \left( \frac{\varphi-u}{k}\right)\!dx, &
\end{align*}
where the domain of integration is comprised in $B(x_{0},r)$. The last integral is negative since $\phi<u$. The first one is
non-negative due to the elementary inequality
\[
	\big\langle |b|^{p-2}b-|a|^{p-2}a, b-a \big\rangle \ge 0,
\]
which holds for all $p>1$ because of the convexity of the $p$-th power. We can take $p=p(x)$.
It follows that $\varphi\le u$ in $B(x_{0},r)$. This contradicts $\varphi(x_{0})> u(x_{0})$. Thus the antithesis
was false and $u$ is a viscosity supersolution.

In a similar way we can prove that $u$ is also a viscosity subsolution.
\end{proof}

\section{Passage to infinity}\label{sec:passage}

Let us study the procedure when $jp(x)\to\infty$. The distance function
\[
	\delta(x) = {\rm dist} (x,\partial\Omega)
\]
plays a crucial role. We write
\begin{equation}
\label{eq:first}
\Lambda_{\infty}= \frac{\|\nabla \delta \|_{\infty}}{\|\delta \|_{\infty}} =  \frac{1}{R}
\end{equation}
where $R$ is the radius of the largest ball inscribed in $\Omega$, the so-called {\em inradius}.
Recall that $\delta$ is Lipschitz continuous and $|\nabla \delta|=1$ a.e. in $\Omega$. 
\vskip.2cm
In fact, $\Lambda_{\infty}$
is the minimum the Rayleigh quotient in the $\infty$-norm:
\begin{equation}
\label{eq:varinfty}
\Lambda_{\infty} = \min_{u}\frac{\|\nabla u\|_{\infty}}{\| u \|_{\infty}},
\end{equation}
where the minimum is taken among all $u\in W^{1,\infty}_{0}(\Omega)$. To see this, let $\xi\in\partial\Omega$
be the closest boundary point to $x\in \Omega$. By the mean value theorem
\[
|u(x)| = |u(x)-u(\xi)| \le \|\nabla u\|_{\infty}|x-\xi| = \|\nabla u\|_{\infty}\delta(x).
\]
It follows that
\[
\Lambda_{\infty} = \frac{1}{\|\delta\|_{\infty}} \le \frac{\|\nabla u\|_{\infty}}{\|u\|_{\infty}}.
\]
\vskip.2cm
Consider
\begin{equation}
\label{eq:lambdaj}
\Lambda_{jp(x)} = \min_{v} \frac{\|\nabla v\|_{jp(x)}}{\|v\|_{jp(x)}},\qquad (j=1,2,3\ldots)
\end{equation}
where the minimum is taken over all $v$ in $C(\overline{\Omega})\cap W^{1,jp(x)}_{0}(\Omega)$.
When $j$ is large, the minimizer $u_{j}$ (we do mean $u_{jp(x)}$) is continuous up to the boundary
and ${u_{j}}_{|\partial\Omega}=0$. This is a property of the Sobolev space.

\begin{prop}
\label{eq:asymptlambda}
\begin{equation}\label{eq:label}
\lim_{j\to\infty}\Lambda_{jp(x)} = \Lambda_{\infty}.
\end{equation}
\end{prop}

\begin{proof}
Assume for simplicity that
\[
	\int_{\Omega} \frac{dx}{p(x)}=1.
\]	
The H\"older inequality implies that
\[
	\| f \|_{jp(x)} \le \| f \|_{lp(x)},\qquad l\ge j.
\]
Let $u_{j}$ be the minimizer in the Rayleigh quotient with the $jp(x)$-norm normalized so that $\| u_{j}\|_{jp(x)}=1$.
Thus,
\[
	\Lambda_{jp(x)} = \|\nabla u_{j}\|_{jp(x)}.
\]
Since $\Lambda_{jp(x)}$ is the minimum, we have
\[
	\Lambda_{jp(x)} \le \frac{\|\nabla \delta\|_{jp(x)}}{\| \delta \|_{jp(x)}},
\]
for all $j=1,2,3\ldots$ Then, by Lemma~\ref{lm:elem}, 
\[
	\limsup_{j\to\infty}\Lambda_{jp(x)} \le \frac{\| \nabla \delta \|_{\infty}}{\|\delta \|_{\infty}} = \Lambda_{\infty}.
\]
It remains to prove that
\[
	\liminf_{j\to\infty} \Lambda_{jp(x)} \ge \Lambda_{\infty}.
\]
To this end, observe that the sequence
$ \| \nabla u_{j} \|_{jp(x)} $ is bounded. Using a diagonalization procedure we can extract a subsequence
$u_{j_{\nu}}$ such that $u_{j_{\nu}}$ converges strongly in each fixed $L^{q}(\Omega)$
and $\nabla u_{j_{\nu}}$ converges weakly in each fixed $L^{q}(\Omega)$. In other words,
\[
	u_{j_{\nu}} \to u_{\infty},\qquad \nabla u_{j_{\nu}} \rightharpoonup \nabla u_{\infty}, \qquad\qquad \text{as   }
		\nu\to\infty,
\]
for some $u_{\infty}\in W^{1,\infty}_{0}(\Omega)$. By the lower semicontinuity of the norm under weak
convergence
\[
 \| \nabla u_\infty \|_q \le \liminf_{\nu\to\infty } \| \nabla u_{j_\nu} \|_q
\]

For large indices $\nu$, we have
\[
\| \nabla u_{j_{\nu}}\|_{q} \le \|\nabla u_{j_{\nu}}\|_{j_{\nu}p(x)} = \Lambda_{j_{\nu}p(x)}.
\]
Therefore,
\[
	\|\nabla u_{\infty}\|_{q} \le \liminf_{\nu\to\infty} \Lambda_{j_{\nu}p(x)}
\]
Finally, letting $q\to\infty$ and taking the normalization into account (by
Ascoli's Theorem, $\|u_{\infty}\|_{\infty}=1$) we obtain
\[
 \frac{\|\nabla u_{\infty}\|_{\infty}}{\|u_{\infty}\|_{\infty}} \le 
\liminf_{\nu\to\infty} \Lambda_{j_\nu p(x)},
\]
but, since $u_{\infty}$ is admissible, $\Lambda_{\infty}$ is less than or equal to the above ratio. This
implies that
\begin{equation*}
  \lim_{\nu\to\infty} \Lambda_{j_\nu p(x)} = \Lambda_\infty.
\end{equation*}
By possibly repeating the above, starting with an arbitrary subsequence of variable exponents, it follows that
the limit \eqref{eq:label} holds for the full sequence. This concludes the proof.
\end{proof}
Using Ascoli's theorem we can assure that the convergence $u_{j_{\nu}}\to u_{\infty}$ is uniform in $\Omega$.
Thus the limit of the normalized first eigenfunctions is continuous and we have
\[
	u_{\infty}\in C(\overline{\Omega})\cap W^{1,\infty}_{0}(\Omega),
\]
with $u_{\infty}{}_{| \partial\Omega}=0$, $u_{\infty}\ge0$, $u_{\infty} \not\equiv0$. However,
{\em the function $u_{\infty}$ might depend on the particular sequence extracted.}

\begin{thm}\label{thm:visc}
The limit of the normalized first eigenfunctions is a viscosity solution of the equation
\[
	\max\left\{ \Lambda_{\infty} - \frac{|\nabla u|}{u}, \Delta_{\infty(x)}\left( \frac{u}{K}\right) \right\}=0,
\]
where $K=\|\nabla u\|_{\infty}$.
\end{thm}

\begin{rmk}
The limit $u$ of the normalized first eigenfunctions is a non-negative function. At the points where $u>0$,
the equation above means that the largest of the two quantities is zero.
At the points\footnote{When $u<0$ this is not the right 
equation, but we keep $u\ge 0$.} where $u=0$, we agree that
first part of the equation is $\Lambda_\infty u = |\nabla u|$.
\end{rmk}

\begin{proof}[Proof of Theorem~\ref{thm:visc}]
We begin with the case of viscosity supersolutions. If $\phi\in C^{2}(\Omega)$ touches $u_{\infty}$ from below
at $x_{0}\in \Omega$, we claim that
\[
\Lambda_{\infty}\le \frac{|\nabla \phi(x_{0})|}{\phi(x_{0})}, \quad\text{and}\quad \Delta_{\infty(x_{0})}
	\left(\frac{\phi(x_{0})}{K}\right)\le 0,
\]
where $K=K_{u_{\infty}}$. We know that $u_{j}$ is a viscosity (super)solution of the equation
\[
	\Delta_{jp(x)}u - |\nabla u|^{jp(x)-2} \ln K_{j} \big\langle \nabla u, j\nabla p(x)\big\rangle
	+\Lambda_{jp(x)}^{jp(x)} S_{jp(x)}|u|^{jp(x)-2}u=0
\]
where $K_{j}=\|\nabla u_{j}\|_{jp(x)}$ and
\[
	S_{jp(x)} = \frac{\displaystyle \int_{\Omega}  \left| \frac{\nabla u_{j}}{K_{j}}\right|^{jp(x)} \,dx}{
		\displaystyle\int_{\Omega}\left|\frac{u_{j}}{k_{j}}\right|^{jp(x)}\,dx}.
\]
We have the trivial estimate
\[
	\frac{p^{-}}{p^{+}} \le S_{jp(x)}\le \frac{p^{+}}{p^{-}}.
\]
We need a test function $\psi_{j}$ touching $u_{j}$ from below at a point $x_{j}$ very near $x_{0}$.
To construct it, let $B(x_{0},2R)\subset\Omega$. Obviously,
\[
	\inf_{B_{R}\setminus B_{r}}\big\{ u_{\infty}-\phi \big\}>0,
\]
when $0<r<R$. By the uniform convergence, 
\[
	\inf_{B_{R}\setminus B_{r}}\big\{ u_{\infty}-\phi \big\}> u_{j}(x_{0})-u_{\infty}(x_{0})=u_{j}(x_{0})-\phi(x_{0}),
\]	
provided $j$ is larger than an index large enough, depending on $r$. For such large indices, $u_{j}-\phi$
attains its minimum in $B(x_{0},R)$ at a point $x_{j}\in B(x_{0},r)$, and letting $j\to\infty$, we see that
$x_{j}\to x_{0}$, as $j\to\infty$.
Actually, $j\to\infty$ via the subsequence $j_{\nu}$ extracted, but we drop this notation.

Define
\[
	\psi_{j} = \phi + \big( u_{j}(x_{j}) - \phi(x_{j}) \big).
\]
This function touches $u_{j}$ from below at the point $x_{j}$. Therefore $\psi_{j}$ will do as a test function for
$u_{j}$. We arrive at
\begin{align}
|\nabla  & \phi (x_{j})|^{jp(x_{j})-4} \bigg\{ |\nabla\phi(x_{j})|^{2}\Delta\phi(x_{j}) & \nonumber \\ 
 &\:\: +\left( jp(x_{j})-2\right) \Delta_{\infty}\phi(x_{j}) + |\nabla\phi(x_{j})|^{2 }\ln \big(|  \nabla \phi(x_{j})|\big) 
 \Big\langle \nabla\phi(x_{j}),j\nabla p(x_{j})\Big\rangle\bigg\} & \nonumber \\
 &\le -\Lambda_{jp(x_{j})}^{jp(x_{j})} S_{jp(x_{j})} |\phi(x_{j})|^{jp(x_{j})-2}\phi(x_{j})  & \nonumber \\
  \label{align} & \qquad \qquad \qquad  + |\nabla \phi(x_{j})|^{jp(x_{j})-2} \ln K_{j} \Big\langle \nabla \phi(x_{j}),j\nabla p(x_{j})\Big\rangle .&
\end{align}
First, we consider the case $\nabla \phi(x_0)\neq 0$. Then $\nabla \phi(x_j)\neq 0 $ for large indices.
Dividing by
$$(jp(x_{j})-2)|\nabla\phi(x_{j})|^{jp(x_{j})-2}$$ we obtain
\begin{align*}
& \frac{|\nabla\phi(x_{j})|^{2}\Delta \phi(x_{j}) }{jp(x_{j})-2} + \Delta_{\infty}\phi(x_{j})
+ |\nabla\phi(x_{j})|^{2}\,\ln|\nabla\phi(x_{j})|\,\left\langle \nabla\phi(x_{j}), \frac{\nabla p(x_{j})}{p(x_{j})-2/j}\right\rangle
& \nonumber\\
& \qquad\le \ln K_{j}\left\langle \nabla\phi(x_{j}), \frac{\nabla p(x_{j})}{p(x_{j})-2/j}\right\rangle 
-\left(  \frac{ \Lambda_{jp(x_{j})} \phi(x_{j})}{|\nabla\phi(x_{j})|} \right)^{jp(x_{j})-4 } \Lambda_{jp(x)}^{4} S_{jp(x_{j})} \phi(x_{j})^{3}. & 	
\end{align*}
In this inequality, all terms have a limit except possibly the last one. In order to avoid a contradiction, we must have
\begin{equation}\label{eq:convcond}
 		\limsup_{j\to\infty} \frac{\Lambda_{jp(x_{j})} \phi(x_{j})}{|\nabla\phi(x_{j})|}\le1.
\end{equation}
Therefore
\begin{equation}\label{eq:alte1}
	\Lambda_{\infty}\phi(x_{0}) - |\nabla\phi(x_{0})| \le 0,
\end{equation}
as desired. Taking the limit we obtain
\[
\Delta_{\infty}\phi(x_{0}) + |\nabla\phi(x_{0})|^{2}\ln \left\vert\frac{\nabla\phi(x_{0})}{K_{\infty}}\right\vert
	\Big\langle\nabla\phi(x_{0}),\nabla\ln p(x_{j})\Big\rangle \le 0.
\]

Second, consider the case $\nabla\phi(x_0)=0$. Then the last inequality above is evident. Now the inequality
\[
 \Lambda_\infty \phi(x_0) - |\nabla\phi(x_0)| \le 0
\]
reduces to $\phi(x_0)\le 0$. 
But, if $\phi(x_0)>0$, then $\phi(x_j) \neq 0$ for large indices. According to inequality \eqref{align}
we must have $|\nabla \phi(x_j)| \neq 0$ and so we can divide by $(jp(x_j)-2)|\nabla\phi(x_j)|^{jp(x_j)-2}$ and conclude from \eqref{eq:convcond} that $\phi(x_{0})=0$, in fact.
This shows that we have a viscosity supersolution. 

In the case of a subsolution one has to show that for a test function $\psi$ touching $u_{\infty}$ from above at $x_{0}$ at least
one of the inequalities
\[
	\Lambda_{\infty} \psi_{\infty}(x_{0}) -|\nabla\psi(x_{0})|\ge 0
\]
or 
\[
	\Delta_{\infty}\psi(x_{0})  + |\nabla\psi(x_{0})|^{2} \ln\left\vert \frac{\nabla \psi(x_{0})}{K_{\infty}}\right\vert
		\Big\langle \nabla\psi(x_{0}),\nabla\ln p(x_{0}) \Big\rangle\ge0
\]
is valid. We omit this case, since the proof is pretty similar to the one for supersolutions.
\end{proof}
\section{Local uniqueness}

The {\em existence} of a viscosity solution to the equation
\[
  \max\left\{ \Lambda_\infty - \frac{|\nabla u|}{u}, \Delta_{\infty(x)} \left( \frac{u}{\|\nabla u \|_\infty}\right) \right\}=0
\]
was established in section~\ref{sec:passage}. The question of {\em uniqueness} is a more delicate one.

In the special case of a constant exponent, say $p(x)=p$, there is a recent counterexample in~\cite{HSY} of a domain (a dumb-bell shaped one) in which there
are several linearly independent solutions in $C(\overline{\Omega})\cap W^{1,\infty}_0(\Omega)$ of the equation
\[
  \max\left\{ \Lambda - \frac{|\nabla u|}{u},\Delta_\infty u \right\} =0, \qquad \Lambda = \Lambda_\infty.
\]
It is decisive that they have boundary values zero. According to~\cite[Theorem 2.3]{JLM}, this cannot happen for strictly positive boundary values, which excludes eigenfunctions.
This partial uniqueness result implied that there are no positive eigenfunctions for $\Lambda\neq \Lambda_\infty$, cf.~\cite[Theorem 3.1]{JLM}.

Let us return to the variable exponents. Needless to say, one cannot hope for more than in the case of a constant exponent. Actually, a condition
involving the quantities $\min u $, $\max u$, $\max |\nabla \ln p|$ taken over subdomains enters. This complicates the matter and restricts the result.

We start with a normalized positive viscosity solution $u$ of the equation
\begin{equation}\label{eq:u}
  \max\left\{ \Lambda_\infty - \frac{|\nabla u|}{u},\, \Delta_{\infty(x)} u \right\}=0.
\end{equation}
Now $K = \|\nabla u\|_\infty=1$. The normalization is used in no other way than that the constant $K$ is erased. This equation is not a ``proper'' one\footnote{A term used in the viscosity theory for second order equations} and
the first task is to find the equation for $v=\ln(u)$.

\begin{lm}
 \label{lm:auxeqn}
Let $C>0$. The function
\[
    v = \ln (C u) 
\]
is a viscosity solution of the equation
\begin{align}\label{eq:auxeqn}
 \max \bigg\{ \Lambda - |\nabla v |, \Delta_\infty v + |\nabla v|^2  \ln\left(\frac{|\nabla v|}{C}\right) \langle \nabla v,\nabla\ln p\rangle
   + v|\nabla v|^2  \langle \nabla v, \nabla \ln p \rangle \bigg\} =0 . 
\end{align}
\end{lm}
 
We need a {\em strict} supersolution (this means that the $0$ in the right hand side has to be replaced by a negative quantity) which approximates $v$ uniformly. To this end we use the
{\em approximation of unity} introduced in~\cite{JLM}. Let
\[
  g(t) = \frac{1}{\alpha} \ln\big( 1+A (e^{\alpha t}-1) \big), \qquad A>1, \: \alpha>0,
\]
and keep $t>0$. The function
\[
  w = g(v) 
\]
will have the desired properties, provided that $v\ge 0$. This requires that
\[
    C u(x) \ge 1,
\]
which cannot hold globally for an eigenfunction, because $u=0$ on the boundary. This obstacle restricts the method to local considerations.
We are forced to limit our constructions to subdomains.

We use a few elementary results:
\begin{align}
 & 0 < g(t) -t < \frac{A-1}{\alpha}, & \nonumber \\ & A^{-1}(A-1) e^{-\alpha t} < g'(t) -1 < (A-1) e^{-\alpha t}, &\nonumber \\ \label{eq:qproperty}
& g(t) -t < \frac{A}{\alpha} (e^{\alpha t} -1 )(g'(t) -1), &  \\ & g''(t) = -\alpha (g'(t)-1)g'(t), & \nonumber \\ & 0 < \ln g'(t) < g'(t) -1. &\nonumber
\end{align}
In particular, $g'(t) -1$ will appear as a decisive factor in the calculations. The formula
\begin{equation*}
	\ln g'(t) = \ln A -\alpha (g(t)-t)
\end{equation*}
is helpful.

We remark that in the next lemma our choice of the parameter $\alpha $ is not optimal, but it is necessary to take $\alpha>1$, at least. For convenience, we set $\alpha=2$.

\begin{lm}\label{lm:approx}
 Take $\alpha=2$ and assume that $1<A<2$. If $v>0$ is a viscosity supersolution of equation~\eqref{eq:auxeqn}, then $w=g(v)$ is a viscosity supersolution of the equations
\[
    \Lambda - \frac{|\nabla w|}{g'(v)}=0,
\]
and
\[
  \Delta_\infty w + |\nabla w |^2 \ln \left( \frac{\nabla w}{C} \right) \langle \nabla w,\nabla \ln p \rangle + w |\nabla w|^2 \langle \nabla w , \nabla \ln p \rangle  + |\nabla w|^4 = -\mu,
\]
where
\[
 \mu = A^{-1} (A-1) |\nabla w|^3 e^{-2v}  \Big\{ \Lambda - \| e^{2v } \nabla \ln p \|_\infty \Big\},
\]
provided that
\[
    \| e^{2v }\nabla \ln p \|_\infty < \Lambda.
\]
\end{lm}

\begin{rmk}
 \label{rmk:constmyu}
We can further estimate $\mu$ and replace it by a constant, viz.
\[
 A^{-1} \Lambda^3 (A-1) e^{-2 \| v \|_\infty } \big\{ \Lambda  - \| e^{2v} \nabla \ln p \|_\infty \big\},
\]
but we prefer not to do so.
\end{rmk}

\begin{proof} The proof below is only formal and should be rewritten in terms of test functions. One only has to observe that an arbitrary test function  $\varphi$ touching $w$
 from below can be represented as $\varphi = g(\phi) $ where $\phi$ touches $v$ from below.

First we have the expressions

\begin{align*}
 & \nabla w = g'(v) \nabla v ,& \\
& \Delta_\infty w = g'(v)^2 g''(v) |\nabla v|^4 + g'(v)^3 \Delta_\infty v, & \\
& |\nabla w |^2 \ln\left( \frac{|\nabla w|}{C}\right) \langle \nabla w,\nabla \ln p\rangle & \\
& \quad = g'(v)^3 \Big\{ |\nabla v|^2 \ln \left(\frac{|\nabla v|}{C}\right)\langle \nabla v,\nabla \ln p\rangle + |\nabla v|^2 \ln(g'(v)) \langle \nabla v,\nabla\ln p\rangle. &
\end{align*}
Then, using that $v$ is a supersolution, we get 
\begin{align*}
 \Delta_\infty & w +|\nabla w |^2 \ln\left( \frac{|\nabla w|}{C}\right) \big\langle \nabla w,\nabla \ln p\,\big\rangle & \\
 & = g'(v)^2g''(v)|\nabla v|^4 + g'(v)^3 \Big\{ \Delta_\infty v + |\nabla v|^2 \ln\left(\frac{|\nabla v|}{C}\right) \big\langle \nabla v,\nabla \ln p\,\big\rangle \Big\} & \\
 & \qquad + g'(v)^3|\nabla v|^2 \ln (g'(v)) \Big\langle \nabla v,\nabla \ln p \,\Big\rangle & \\
& \le g'(v)^2g''(v)|\nabla v|^4 + g'(v)^3 \Big\{ -v|\nabla v|^2\big\langle \nabla v,\nabla\ln p\,\big\rangle - |\nabla v|^4 \Big\} & \\
 & \qquad + g'(v)^3|\nabla v|^2 \ln (g'(v)) \Big\langle \nabla v,\nabla \ln p \,\Big\rangle. &
\end{align*}
Let us collect the terms appearing on the left-hand side of the equation for $w$. Using the formulas \eqref{eq:qproperty} for $g''(v)$ and $\ln\big(g'(v)\big)$ we arrive at
\begin{align*}
  \Delta_\infty w & +|\nabla w |^2 \ln\left( \frac{|\nabla w|}{C}\right) \big\langle \nabla w,\nabla \ln p\,\big\rangle +|\nabla w|^4  + w |\nabla w|^2\,\Big\langle \nabla w,\nabla \ln p\Big\rangle, & \\
& \le g'(v)^3 |\nabla v|^3 \big( g'(v)-1\big) \Big\{ -|\nabla v| + |\nabla \ln p| \Big\} + g'(v)^3 |\nabla v|^3 \big(g(v)-v\big) |\nabla\ln p|, &
\end{align*}
after some arrangements. Using
\[
 g(t)-t < \frac{A}{2} ( e^{2t}-1) (g'(t)-1) \le (e^{2t}-1)(g'(t)-1),
\]
and collecting all the terms with the factor $|\nabla\ln p|$ separately, observing that $1+(e^{2t}-1) = e^{2t}$, we see that the right-hand side is less than
\begin{align*}
&   g'(v)^3 |\nabla v|^3 \big( g'(v)-1\big) \{ -|\nabla v| + |e^{2v}\nabla \ln p| \}
\le |\nabla w|^3 A^{-1} (A-1) e^{-2v} \{ - \Lambda + |e^{2v}\nabla \ln p| \}, &
\end{align*}
since the expression in braces is negative.
\end{proof}

We abandon the requirement of zero boundary values. Thus $\Omega$ below can represent a proper subdomain.
Eigenfunctions belong to a Sobolev space 
but we cannot ensure this for an arbitrary viscosity
solution. This requirement is therefore included in our next theorem.

\begin{thm}\label{thm:subduniq}
Suppose that $u_1\in C(\overline{\Omega})$ is a viscosity subsolution
and that $u_2\in C(\overline{\Omega})$ is a viscosity supersolution of equation~\eqref{eq:u}.
Assume that at least one of them belongs to $W^{1,\infty}(\Omega)$.
If $u_1(x)>0$ and $u_2(x)\ge m_2>0$ in $\Omega$, and
\begin{equation}\label{eq:bd}
  3 \left \| \left(\frac{u_2}{m_2}\right)^2 \nabla \ln p \right\|_\infty \le \Lambda,
\end{equation}
then the following comparison principle holds:
\[
  u_1\le u_2 \quad \text{on   }\partial\Omega\quad  \Longrightarrow \quad u_1\le u_2 \quad \text{in   }\Omega.
\]
\end{thm}

\begin{proof}
 Define
\[
 v_1 = \ln ( C u_1),\qquad v_2 = \ln(C u_2) ,
\]
with $C=1/m_2$. Then $v_2>0$, but $v_1$ may take negative values. We define
\[
  w_2 = g(v_2), \quad \alpha=2,\quad 1<A<2.
\]
If $v_2\ge v_1$, we are done. If not, consider the open subset $\{v_2<v_1\}$ and denote
\[
    \sigma = \sup \big\{ v_1-v_2\big\} >0.
\]
Note that $\sigma$ is independent of $C$. (The antithesis was that $\sigma>0$.) Then, taking $A=1+\sigma$,
\[
  v_2<w_2<v_2+ \frac{A-1}{2} = v_2 + \frac{\sigma}{2}. 
\]
Note that $v_{1}-w_{2}=v_{1}-v_{2}+v_{2}-w_{2}\ge v_{1}-v_{2} -\sigma/2$. Taking the supremum on
the subdomain $\mathcal{U} = \{ w_2 < v_1\} $ we have 
\[
  \sup_\mathcal{U} \big\{ v_1 - w_2 \} \ge \frac{\sigma}{2} > 0=\max_{\partial \mathcal{U}} \big\{ v_1-w_2\big\} 
\]
and $\mathcal{U}\Subset \Omega$, i.e. $\mathcal{U}$ is strictly interior. Moreover,
\begin{equation}
 \label{eq:3sigma}
\sup \big\{ v_1- w_2 \big\}  \le \frac{3\sigma}{2}.
\end{equation}

In order to obtain a contradiction, we double the variables and write 
\[
  \mathsf{M}_j = \max_{\overline{\mathcal{U}}\times\overline{\mathcal{U}}} \left\{ v_1(x)-w_2(y) - \frac{j}{2}|x-y|^2 \right\}.
\]
If the index $j$ is large, the maximum is attained at some interior point $(x_j,y_j)$ in $\mathcal{U}\times\mathcal{U}$. The points converge to some 
interior point, say $x_j\to\hat{x}$, $y_j\to\hat{x}$, and
\[
    \lim_{j\to\infty} j |x_j-y_j|^2 =0.
\]
This is a standard procedure. According to the ``Theorem of Sums``, cf.~\cite{CIL} or~\cite{K}, there exist symmetric $n\times n$-matrices
$\mathbb{X}_j$ and $\mathbb{Y}_j$ such that
\begin{align*}
 & \Big( j(x_j-y_j), \mathbb{X}_j \Big) \in \overline{J^{2,+}_\mathcal{U}} v_1(x_j), & \\
 & \Big( j(x_j-y_j), \mathbb{Y}_j \Big) \in \overline{J^{2,-}_\mathcal{U}} w_2(y_j) ,& \\
 & \Big\langle \mathbb{X}_j \xi\, , \, \xi \Big\rangle \le \Big\langle \mathbb{Y}_j\xi\,,\,\xi\Big\rangle, \quad \text{when  } \xi\in \mathbb{R}^n.
\end{align*}
The definition of the semijets and their closures $\overline{J^{2,+}_\mathcal{U}}$, $\overline{J^{2,-}_\mathcal{U}}$ can be found in the above mentioned
references\footnote{Symbolically the interpretation is: $j(x_j-y_j)$ means $\nabla v_1(x_j)$ and $\nabla w_2(y_j)$, $\mathbb{X}_j$ means $D^2v_1(x_j)$, and
$\mathbb{Y}_j$ means $D^2w_2(y_j)$.}. The equations have to be written in terms of jets.

We exclude one alternative from the equations. In terms of jets
\[
  \Lambda - \frac{|\nabla w_2|}{g'(v_2)} \le 0 \qquad \text{reads} \qquad \Lambda - \frac{j|x_j-y_j|}{g'(v_2(y_j))}\le 0
\]
and, since $v_2>0$, $g'(v_2(y_j))>1$, and so
\[
\Lambda < j|x_j-y_j|. 
\]
This rules out the alternative $\Lambda- |\nabla v_1(x_j)|\ge 0$ in the equation for $v_1$, which reads
$ \Lambda - j|x_j-y_j|\ge 0 $. Therefore we must have that \mbox{$\Delta_\infty v_1 + \dots + |\nabla v_1|^4\ge 0$},
i.e.
\begin{align*}
 & \Big\langle \mathbb{X}_j \,j(x_j-y_j),j(x_j-y_j)\Big\rangle + j^2|x_j-y_j|^2 \ln\left(\frac{j|x_j-y_j|}{C}\right) \Big\langle  j(x_j-y_j),\nabla\ln p(x_j)\Big\rangle & \\
  & \qquad + v_1(x_j) j^2 |x_j-y_j|^2 \ln\left( \frac{j|x_j-y_j|}{C} \right) \Big\langle j(x_j -y_j), \nabla \ln p(x_j)  \Big\rangle  + j^4|x_j-y_j|^4\ge 0.&
\end{align*}
The equation for $w_2$ reads
\begin{align*}
 & \Big\langle \mathbb{Y}_j \,j(x_j-y_j),j(x_j-y_j)\Big\rangle + j^2|x_j-y_j|^2 \ln\left(\frac{j|x_j-y_j|}{C}\right) \Big\langle  j(x_j-y_j),\nabla\ln p(y_j)\Big\rangle & \\
  & \qquad + w_2(y_j) j^2 |x_j-y_j|^2 \ln\left( \frac{j|x_j-y_j|}{C} \right) \Big\langle j(x_j -y_j), \nabla \ln p(y_j) \Big\rangle + j^4|x_j-y_j|^4  & \\
  & \qquad \le -A^{-1} \sigma\, j^3|x_j-y_j|^3 e^{-2v_2(y_j)} \Big\{ \Lambda - \left\| e^{2v_2}\nabla \ln p\right\|_{\infty,\,\mathcal{U}} \Big\}. &
\end{align*}
Subtracting the last two inequalities, we notice that the terms $j^4|x_j-y_j|^4$ cancel. The result is
\begin{align*}
   \Big\langle \big( & \mathbb{Y}_j-\mathbb{X}_j\big) \, j(x_j-y_j), j(x_j-y_j)\Big\rangle & \\
  & + j^2|x_j-y_j|^2 \ln\left(\frac{j|x_j-y_j|}{C}\right) \Big\langle j (x_j-y_j), \nabla\ln p(y_j)-\nabla \ln p(x_j) \Big\rangle & \\
  & + j^2|x_j-y_j|^2 \Big\langle j(x_j-y_j), w_2(y_j)\, \nabla\ln p(y_j) - v_1(x_j) \nabla \ln p(x_j) \Big\rangle & \\
  & \le -A^{-1} \sigma \, j^3|x_j-y_j|^3 \,e^{-2v_2(y_j)} \Big\{ \Lambda - \left\| e^{2v_2} \nabla \ln p \right\|_{\infty,\,\mathcal{U}} \Big\}. &
\end{align*}
The first term, the one with matrices, is non-negative and can be omitted from the inequality. Then we move the remaining terms and divide by $j^3|x_j-y_j|^3$ to get
\begin{align*}
 & A^{-1}   \,\sigma \, e^{-2v_2(y_j)} \Big\{ \Lambda - \left\| e^{2v_2} \nabla \ln p \right\|_{\infty,\,\mathcal{U}} \Big\} & \\
 & \,\le \left\vert \ln \frac{j|x_j-y_j|}{C}\right\vert \, \big\vert \nabla\ln p(y_j)-\nabla \ln p(x_j) \big\vert 
  + \big\vert w_2(y_j) \nabla\ln p(y_j) - v_1(x_j) \nabla \ln p(x_j) \big\vert
&
\end{align*}

We need the uniform bound
\[
  \Lambda \le j|x_j-y_j| \le L.
\]
The inequality with $\Lambda$ was already clear. We can take $L=2 \|v_1\|_{\infty,\,\mathcal{U}}$
\sloppy or \mbox{$L=\|w_2\|_{\infty,\,\mathcal{U}}\le 4 \| v_2\|_{\infty,\, \mathcal{U}}$}, using the definition of $\mathsf{M}_j$.
Taking the limit as $j\to\infty$ we use the continuity of $\nabla\ln p$ to arrive at
\begin{align*}
 A^{-1}   \,\sigma \, \Big\{ \Lambda - \left\| e^{2v_2} \nabla \ln p \right\|_{\infty,\,\mathcal{U}} \Big\} 
  &  \le e^{2v_2(\hat{x})} \big\vert w_2(\hat{x}) \nabla\ln p(\hat{x}) - v_1(\hat{x})\nabla \ln p(\hat{x}) \big\vert. & 
\end{align*}
Recall \eqref{eq:3sigma}. Since $A=1+\sigma$, the above implies that
\[
 A^{-1}   \,\sigma \, \Big\{ \Lambda - \left\| e^{2v_2} \nabla \ln p \right\|_{\infty,\,\mathcal{U}} \Big\}    \le \| e^{2v_2} 
 \nabla \ln p\|_{\infty,\,\mathcal{U}} \: \frac{3\sigma}{2}.
\]
Divide out $\sigma$. Now $A^{-1}\ge1/2$. The final inequality is
\[
  \Lambda \le 3 \| e^{2v_2}\nabla \ln p  \|_{\infty,\,\mathcal{U}}.
\]
Thus there is a contradiction, if the opposite inequality is assumed to be valid. Recall that 
\[
    e^{2v_1} = \left( \frac{u_2}{m_2}\right)^2
\]
to finish the proof.

\end{proof}

\begin{cor}
 Local uniqueness holds. In other words, in a sufficiently small interior subdomain we cannot perturb the eigenfunction continuously.
\end{cor}
\begin{proof}
 We can make
\[
    \frac{\displaystyle \max_{\mathcal{U}} u}{\displaystyle \min_{\mathcal{U}}u}
\]
as small as we please, by shrinking the domain $\mathcal{U}$. Thus condition \eqref{eq:bd} is valid with the
$L^\infty$ norm taken over $\mathcal{U}$.
\end{proof}

\section{Discussion about the one-dimensional case}
In the one-dimensional case an explicit comparison of the minimization problem for the two Rayleigh quotients
\eqref{eq:luxray} and \eqref{eq:rayp} is possible. Let $\Omega=(0,1)$ and consider the limits of the problem
coming from minimizing either
\[\tag{I}
\frac{\| u'\|_{jp(x)}}{\|u\|_{jp(x)}}
\]
or
\[\tag{II}
\frac{\displaystyle\int_{0}^{1}|v'(x)|^{jp(x)}\,dx}{\displaystyle\int_{0}^{1}|v(x)|^{jp(x)}\,dx},\qquad
\text{with   }\int_{0}^{1}|v(x)|^{jp(x)}\,dx=C,
\]
as $j\to\infty$. In the second case the equation is
\[
\min\left\{ \Lambda - \frac{|v'|}{v}, (v')^{2}v'' + (v')^{3} \ln(|v'|) \frac{p}{p'}\right\} = 0
\]
for $v>0$ ($v(0)=0$, $v(1)=0$, $\|v^{p}\|_{\infty}=C$).

The Luxemburg norm leads to the same equation, but with
\[
	v(x) = \frac{u(x)}{\|u'\|_{\infty}} =\frac{u(x)}{K}
\]
as in equation \eqref{eq:curiouseig}. Thus all the solutions violating the condition $\|v'\|_{\infty}=1$ are ruled out.
This is the difference between the two problems.
\vskip.2cm
Let us return to (II). The equation for $v$ (without any normalization) can be solved. Upon separation of variables, we obtain
\begin{equation*}
v ( x)  = \begin{cases}
\displaystyle\int_{0}^{x} e^{\frac{A}{p(t)}}\,dt,& \quad \text{\rm when   } 0\le x \le x_{0}, \\ & \\
\displaystyle\int_{x}^{1} e^{\frac{A}{p(t)}}\,dt, & \quad \text{\rm when   } x_{0} \le x \le 1, 
\end{cases}
\end{equation*}
where the constant $A$ is at our disposal and the point $x_{0}$ is determined by the continuity condition
\[
	\int_{0}^{x_{0}} e^{\frac{A}{p(t)}}\,dt=\int_{x_{0}}^{1} e^{\frac{A}{p(t)}}\,dt.
\]
Clearly, $0<x_{0}<1$. Now $\Lambda$ is determined from
\[
	\frac{v'(x_{0}^{-})}{v(x_{0})} = \Lambda =  -\frac{v'(x_{0}^{+})}{v(x_{0})}.
\]
Provided that the inequality
\[
	\frac{|v'(x)|}{v(x)} \ge \Lambda \quad (0<x<1, x\neq x_{0})
\]
holds, the number $\Lambda$ is an eigenvalue for the non-homogeneous problem. What about the value
of $A$?  Given $C$, we can determine $A$ from
\[
	\max_{0<x<1} v(x)^{p(x)}=C.
\]
At least for a suitable $p(x)$, we can this way reach any real number $A$ and therefore $\Lambda$ can take all positive
values, as $C$ varies.

The problem in the Luxemburg norm is different. If $u$ is an eigenfunction and
\[
	v = \frac{u}{\|u'\|_{\infty}},
\]
then $0\le v'(x)\le 1$ in some interval $(0,x_{0})$. But the equation leads to
\[
	\frac{u'(x)}{\|u'\|_{\infty}} = e^{-\frac{A_{1}}{p(x)}}, \quad A_{1}\ge0,
\]
in $(0,x_{0})$ and 
\[
	-\frac{u'(x)}{\|u'\|_{\infty}} = e^{-\frac{A_{2}}{p(x)}}, \quad A_{2}\ge0,
\]
in $(x_{0},1)$. (In fact, $A_{1}=A_{2}$). But this is impossible at points where the left-hand side is $\pm1$, unless
at least one of the constants $A_{1},A_{2}$ is zero, say that $A_{1}=0$. Then $u(x)=x$ when $0\le x\le x_{0}$.
The determination of $\Lambda$ from the equation
\[
	\frac{1}{x_{0}} = \Lambda = \frac{e^{-A_{2}/ p(x_{0})}}{x_{0}},
\]
forces also $A_{2}=0$. It follows that
\[
	u(x)=\delta(x), \quad \Lambda=\Lambda_{\infty}=2
\]
is the only positive solution of the equation \eqref{eq:curiouseig}. In this problem $\Lambda$ is unique. Recall that $\delta$ is the distance function.

\end{document}